\documentclass[english,reqno]{amsart}
\usepackage{etex}
\usepackage{amsmath,amssymb,amsthm,bbm,mathtools,comment}
\usepackage[shortlabels]{enumitem}
\usepackage[pdftex,colorlinks,backref=page,citecolor=blue]{hyperref}
\hypersetup{pdfpagemode=UseNone,pdfstartview={XYZ null null 1.00}}
\usepackage[mathscr]{euscript}
\usepackage{tikz,babel,adjustbox}
\usetikzlibrary{decorations.pathreplacing,calc,arrows,arrows.meta,patterns}
\usepackage[final]{microtype}
\usepackage[numbers]{natbib}
\usepackage{cmtiup}
\usepackage{amsfonts}
\usepackage{graphicx}
\usepackage{caption}
\usepackage{verbatim}
\usepackage{array}
\usepackage[frame,cmtip,arrow,matrix,line,graph,curve]{xy}
\usepackage{graphpap, color, pstricks}
\usepackage{pifont}
\usepackage[final]{microtype}
\usepackage{cmtiup}

\allowdisplaybreaks

\theoremstyle{plain}
\newtheorem{theorem}{Theorem}[section]
\newtheorem{lemma}[theorem]{Lemma}
\newtheorem{claim}[theorem]{Claim}
\newtheorem{proposition}[theorem]{Proposition}

\newtheorem{conjecture}[theorem]{Conjecture}

\newcommand{\NN}{\mathbb{N}}

\newcommand{\pr}{\mathbb{P}}

\newcommand{\E}[0]{\mathbb{E}}
\newcommand{\sm}[0]{\setminus}
\renewcommand{\dots}[0]{,\ldots,}

\newcommand{\beq}[1]{\begin{equation}\label{#1}}
\newcommand{\enq}[0]{\end{equation}}

\newcommand{\f}[0]{{\mathcal F}}

\newcommand{\g}[0]{{\mathcal G}}

\newcommand{\M}[0]{{\mathscr M}}

\newcommand{\R}[0]{{\mathcal R}}

\newcommand{\Ra}[0]{\Rightarrow}

\newcommand{\0}[0]{\emptyset}

\newcommand{\C}[2]{{\binom{#1}{#2}}}

\newcommand{\eps}[0]{\varepsilon }

\newcommand{\rred}[1]{\textcolor{red}{#1}}

\newcommand{\mn}[0]{\medskip\noindent}
\newcommand{\nin}[0]{\noindent}
\newcommand{\sH}[0]{{\mathscr H}}

\newcommand{\sub}[0]{\subseteq}

\title{The threshold for the square of a Hamilton cycle}

\author{Jeff Kahn}
\address{Department of Mathematics, Rutgers University, Piscataway, NJ 08854, USA}
\email{jkahn@math.rutgers.edu}

\author{Bhargav Narayanan}
\address{Department of Mathematics, Rutgers University, Piscataway, NJ 08854, USA}
\email{narayanan@math.rutgers.edu}

\author{Jinyoung Park}
\address{ School of Mathematics, Institute for Advanced Study, Princeton, NJ 08540, USA}
\email{jpark@math.ias.edu}

\date{14 September, 2020}
\subjclass[2010]{Primary 05C80; Secondary 05C38}

\begin{document}

\maketitle
\begin{abstract}
Resolving a conjecture of K\"uhn and Osthus from 2012, we show that $p= 1/\sqrt{n}$ is the threshold for the random graph $G_{n,p}$ to contain the square of a Hamilton cycle.   
\end{abstract}
\maketitle

\section{Introduction}
Understanding thresholds for various properties of interest has been central to the study of random graphs since its initiation by Erd\H{o}s and R\'enyi~\citep{ER}, and thresholds for containment of (copies of) specific graphs in the random graphs $G_{n,p}$ and $G_{n,m}$
have been the subject of some of the most powerful work in the area. 
(See e.g.~\citep{BBbook, JLR, FrKa}, to which we also refer for threshold basics.)

Hamilton cycles in random graphs in particular are the subject of an extensive
literature, with, to begin, the question of when they appear posed in~\citep{ER}
and answered in~\citep{ham1, ham2, ham3, ham4}; see \citep{FriezeBib} for a thorough account.
Here we consider a related question first raised by K\"uhn and Osthus~\citep{KO}: 
\emph{when does the square of a Hamilton cycle appear in the random graph?}
(The \emph{k}th \emph{power} of a graph $G$ 
is the graph on $V(G)$ with two vertices joined iff their distance in $G$ is at most $k$.)

For this discussion we write $\sH_n^k$ for the $k$th power of an $n$-vertex cycle
(so a Hamilton cycle of $K_n$).
The expected number of copies of $\sH_n^k$ in $G_{n,p}$ is $(n-1)!p^{kn} / 2$, 
implying that the threshold for appearance of 
$\sH_n^k$ in $G_{n,p}$ (henceforth simply ``threshold for $\sH_n^k$'')
is at least $n^{-1/k}$. 
(We follow a standard abuse in using \emph{``the'' threshold} for 
an order of magnitude rather than a specific value.)
For $k=1$, it was famously shown by P\'osa~\citep{ham1} that the threshold for 
a Hamilton cycle is $\log n / n$---this is driven not by expectation considerations, but by 
the need to avoid isolated vertices---while 
for $k \ge 3$, it follows from a general result of Riordan~\citep{oliver},
based on the second-moment method, that the threshold for $\sH_n^k$ is $n^{-1/k}$. 

The case $k=2$ has proved more stubborn: here there is no obvious analogue of isolated vertices 
pushing the threshold above $n^{-1/2}$, 
but, unlike for larger $k$, 
the second-moment method yields only weak upper bounds.
K\"uhn and Osthus~\citep{KO} conjectured that $n^{-1/2}$ is correct 
and showed that the threshold is $n^{-1/2 + o(1)}$, a bound subsequently 
improved to $(\log n)^4 n^{-1/2}$ by Nenadov and {\v S}kori\'c~\citep{NS}; 
to $(\log n)^3 n^{-1/2}$ by Fischer, {\v S}kori\'c, Steger and Truji\'c~\citep{FSST}; 
and to $(\log n)^2 n^{-1/2}$ in unpublished work of Montgomery (see~\citep{FriezeBib}).
Here we resolve the question, proving the conjecture of~\citep{KO}:

\begin{theorem}\label{mt1}
There is a universal $K$ such that for $p \ge K / \sqrt{n}$,
\[ \pr (G_{n,p} \text{ contains the square of a Hamilton cycle}) \to 1
\,\,\,\,\mbox{as $n \to \infty$}.\]
\end{theorem}
While the aforementioned attempts are all rooted in the notion of `absorption'
introduced in~\citep{RRSz},
the proof of Theorem~\ref{mt1} takes a different approach, based on
the recent resolution, by Frankston and the present authors~\citep{FKNP}, 
of Talagrand's relaxation~\citep{Talagrand} of the `expectation threshold' conjecture of~\citep{KK}. 
We will say (not quite following~\citep{FKNP})
that a 
hypergraph $\g$ on a finite vertex set $V$ is $q$-\emph{spread} if
\beq{spread}
|\g \cap \langle I \rangle| \le q^{|I|}|\g|
\enq
for each $I \subseteq V$, where $\langle I \rangle$ is the increasing family generated by $I$; in this language, the main result of~\citep{FKNP} says that there is a fixed $C$
such that if a hypergraph $\g$ with edges of size at most $\ell$ is $q$-spread, 
then a $(C q \log \ell )$-random subset of $V$ 
is likely to contain some edge of $\g$. 

Applied to the hypergraph 
$\g$ consisting 
of all copies of $\sH_n^2$ (which is $q$-spread with $q\sim\sqrt{e/n}$; see below), 
the result of~\citep{FKNP} says that the threshold for $\sH_n^2$ is at most $\log n / \sqrt{n}$. 
A key point in our proof of Theorem~\ref{mt1}, which eliminates the offending $\log n$, is
the observation that large ``local spreads''
$(|\g\cap \langle I\rangle|/|\g|)^{1/|I|}$ are relatively rare,
a typical value being more like $1/n$ than $1/\sqrt{n}$.

Formally, we prove the following slight weakening of Theorem~\ref{mt1}.
\begin{theorem}\label{mt2}
For each $\eps > 0$ there is a $K $ such that for $p \ge K/\sqrt{n}$,
\[ \pr (G_{n,p} \text{ contains the square of a Hamilton cycle}) \ge 1 - \eps \]
for sufficiently large $n$.
\end{theorem}
\nin
Getting Theorem~\ref{mt1} from this just requires applying the
machinery of Friedgut~\citep{friedgut1, friedgut2} to say 
that the property of containing $\sH_n^2$ has a sharp threshold.
We omit this by now routine step (and the relevant definitions), 
and refer the reader to (e.g.)~\citep{schacht} for a similar argument.

Though there seems little hope of proving such a statement along the present lines, 
it is natural to guess that the above expectation considerations drive the
threshold more precisely, namely:
\begin{conjecture}
For fixed $\eps > 0 $ and $p > (1+ \eps)\sqrt{e  / n}$, 
\[
\pr (G_{n,p} \text{ contains the square of a Hamilton cycle}) \to 1 \,\,\,\,\mbox{as $n \to \infty$}.
\]
\end{conjecture}

The proof of Theorem~\ref{mt2} is given in Section~\ref{s:proof}, with some basic calculations supporting
the argument provided in
Section~\ref{s:pre}.

%%%%%%%%%%%%%%%%
\iffalse
mn
\textbf{Usage.}  This is \rred{all} but we remind:
We use $2^X$ for the collection of subsets of $X$ and $\f\sub 2^X$ is \emph{increasing}
if $A\supseteq B\in \f \Ra A\in \f$.  
A \emph{hypergraph} on (\emph{vertex}) set $V$ is a collection, possibly with repeats, of
subsets of $V$ and is $r$-\emph{uniform} if consists of sets of size $r$.
As usual $G_{n,p}$ is the (``binomial'') random graph on vertex set $[n]:=\{1\dots n\}$ 
in which edges are present independently, each with probability $p$.

\emph{Thresholds}.  
\fi
%%%%%%%%%%%%%%%%

\section{Preliminaries}\label{s:pre}
We will use $\M$ for $E(K_n)$ and from now on write $\sH$ for $\sH_n^2$.
As above, $\g$ is the $(2n)$-uniform hypergraph on vertex set $\M$ consisting of all copies of 
$\sH$ in $K_n$. Thus $|\g| = (n-1)! / 2$, and it is not hard to see 
that $\g$ is $q$-spread with $q = \left[2/(n-1)!\right]^{1/(2n)} \sim \sqrt{e/n}$, meaning (recall)
\beq{spread}
|\g \cap \langle I \rangle|\le q^{|I|}|\g| ~~~\forall  I\sub \M.
\enq

The next two statements implement
the basic idea mentioned above, that large values of $|\g\cap \langle I\rangle|/|\g|$ are rare.

\begin{proposition}\label{easy}
For an $I \subseteq \M$ with $\ell \le n/3$ edges and $c$ components, 
\[ 
|\g \cap \langle I \rangle| \le (16)^\ell \left(n-\left\lceil {\frac{\ell+c}{2}} \right\rceil{-1}\right)!
\]
\end{proposition}

\begin{proposition}\label{easy2}
For an $F\sub \sH$ of size $h$, the number of subgraphs of $F$ with 
$\ell$ edges and $c$ components is at most
\[ (8 e)^\ell \binom{2h}{c}. \]
\end{proposition}
\begin{proof}[Proof of Proposition~\ref{easy}]
Let $I_1 \dots I_c$ be the components of $I$ and $v=|V(I)|$
(where $V(E)$ is the set of vertices used by $E\sub  \M$).
The upper bound on $\ell$ implies that no $I_j$ can 
``wrap around,'' so $|E(I_j)| \le 2|V(I_j)|-3$ for each $j$ and 
\beq{ivc} \ell \le 2v-3c. \enq

We first designate a root vertex $v_j$ 
for each $I_j$ and order $V(I_j)$ by some $\prec_j$
that begins with $v_j$ and in which each $v\neq v_j$ appears later than 
at least one of its neighbors.
We may then bound $|\g \cap \langle I \rangle|$ as follows.

To specify a $J$ ($\in \g$) containing $I$, 
we first specify a cyclic permutation of $\{v_1\dots v_c\}\cup (V(K_n)\sm V(I))$. 
By \eqref{ivc}, the number of ways to do this (namely, $(n-v+c-1)!$) is at most
$
 \left(n-\left\lceil {\frac{\ell+c}{2}} \right\rceil {-1} \right)!
$

We then extend to a full cyclic ordering of $V(K_n)$ (thus determining $J$) by inserting,
for $j=1\dots c$, the vertices of $V(I_j)\sm \{v_j\}$ in the order $~\prec_j$.
This allows at
most four places to insert each vertex (since one of its neighbours has  
been inserted before it and the edge joining them must belong to $J$), so the number of possibilities here
is less than $4^v \le (16)^\ell$, and the proposition follows.
\end{proof}

\begin{proof}[Proof of Proposition~\ref{easy2}]
We need the following standard bound, which 
follows from the fact (e.g.~\citep[p.\ 396, Ex.11]{Knuth}) that the infinite 
$\Delta$-branching rooted tree contains precisely
\[\frac{{{\Delta v} \choose v}}{(\Delta-1)v+1} \le (e\Delta)^{v-1}\] rooted subtrees with $v$ vertices.
\begin{lemma}\label{treelemma}
For a graph $G$ of maximum degree $\Delta$, the number of connected, $h$-edge subgraphs of
$G$ containing a given vertex is less than $(e\Delta)^h$. \qed
\end{lemma}

To specify a subgraph $J$ of $F$ as in Proposition~\ref{easy2}, we proceed as follows. 
We first choose root vertices $v_1\dots v_c$ for the components, say
$J_1\dots J_c$, of $J$, the number of possibilities for this being at most $\C{2h}{c}$.
We then choose the sizes, say 
$\ell_1 \dots \ell_c$, of $J_1\dots J_c$;
here the number of possibilities is at most $\binom{\ell - 1} {c - 1}$
(the number of $c$-compositions of $\ell$, that is,
positive integer solutions of $\ell_1 + \cdots + \ell_c = \ell$). 
Finally, we specify for each $i$
a connected $J_i$ of size $\ell_i$ rooted at $v_i$, which according to  
Lemma~\ref{treelemma} can be done in at most
$\prod(4e)^{\ell_i} = (4e)^\ell$ ways. 
Combining these estimates (with the crude $\binom{\ell - 1} {c - 1}< 2^\ell$)
yields the bound in the proposition.
\end{proof}

\section{Proof of the main result}\label{s:proof}

Recall that $\M=E(K_n)$ and $\g$ is the hypergraph 
of copies of $\sH=\sH_n^2$ in $K_n$, and set $m = |\M| $ ($= \binom{n}{2}$). 

For $S\in \g$ and $X\sub \M$, an \emph{$(S, X)$-fragment} is a set of the form $J \setminus X$ 
with $J \in \g$ contained in $S \cup X$. 
Our main point, Lemma~\ref{MP}, says that for a suitably large $w$,
most pairs $(S,W)$ with $S\in \g$
and $W\in \C{\M}{w}$ admit small fragments.
(We will later need the usual easy transfer of the present discussion to a ``binomial'' $W$.)

Set $k= 4 \sqrt n$ and (for $S,X$ as above) call the pair $(S, X)$ \emph{good} if some $(S,X)$-fragment has
size at most $k$, and \emph{bad} otherwise. 
In what follows we will always assume $S,J\in\g$ and $W\in \C{\M}{w}$,
where $w$ will be $Cn^{3/2}$ for some large constant $C$.

\begin{lemma}\label{MP} 
There is a fixed $C_0 $ such that for all $C \ge C_0$ and $n \in \NN$, 
with $w = Cn^{3/2}$, 
\beq{no.t}
\left| \left\{(S, W) : (S, W) \text{ is bad} \right\} \right| \le 2C^{-k/3} |\g|\C{m}{w}.
\enq

\end{lemma}

\begin{proof}
We may of course assume $n$ is large, since values below any fixed $n_0$ can be
handled trivially by adjusting $C_0$.
It is enough to show 
\beq{given.t}
\left| \left\{ (S,W):(S, W) \text{ is bad}, |W \cap S|=t \right\} \right| 
\le 2C^{-k/3} |\g|\binom{2n}{t}\binom{m-2n}{w-t}
\enq
for $t\in \{0\dots 2n\}$, since summing 
%\eqref{given.t} 
over $t$ then gives \eqref{no.t}.

Now aiming for~\eqref{given.t}, we fix $t$, 
set $w' = w - t$, and bound the number of 
bad $(S,W)$'s with $|W \cap S| = t$ (so $|W \setminus S| = w'$ and $|W\cup S|=w' +2n$). 
Call $Z \in \binom{\M}{w'+2n}$ \emph{pathological} if
\[|\{S \subseteq Z:(S, Z \setminus S) \text{ is bad} \}| 
> C^{-k/3} |\g|\C{m-2n}{w'}\Big/\C{m}{w'+2n} 
=C^{-k/3}|\g | \C{w'+2n}{2n}\Big/\C{m}{2n},\]
and, when $|S\cup X| =w'+2n$, say $(S, X)$ is pathological if $X \cup S$ is.
We bound the nonpathological and pathological parts of \eqref{given.t} separately.

\mn

\textbf{Nonpathological contributions.} 
We claim that the number of nonpathological $(S,W)$'s in \eqref{given.t} is less than
\beq{nonpathest} 
C^{-k/3} |\g | {2n \choose t} {m-2n \choose {w'}}.
\enq
To see this we specify $(S,W)$ by specifying first $Z:= S\cup W$, then $S$, and then $W$.  
The number of possibilities for $Z$ is at most
\[
\C{m}{w'+2n}, 
\]
while, since 
$(S,W)$ bad implies $
%(S, W \setminus S) = 
(S, Z\setminus S)$ bad (and $Z$ is nonpathological),
the number of possibilities for $S$ given $Z$ is at most 
\[ 
C^{-k/3} |\g | \C{m-2n}{w'}\Big/\C{m}{w'+2n};
\]
and of course the number of possibilities for $W$ given $Z$ and $S$ is at most $\C{2n}{t}$.
So we have \eqref{nonpathest}.

\mn

\textbf{Pathological contributions.} The main point here is the following estimate.
(Recall $S, J\in\g$.)

\begin{claim}\label{claim.patho}
For a given $S $, $Y$ chosen uniformly from ${\M \setminus S \choose w'}$, and large enough $C$,
\beq{patho} 
\E\left[ |\{J \subseteq Y \cup S:|J \cap S| \ge k\}| \right] 
\le C^{-2k/3} |\g | \C{w'+2n}{2n}\Big/{m \choose 2n}.
\enq
\end{claim}
This is proved below.
Assuming for the moment it is true, we show that the number of pathological $(S,W)$'s 
in \eqref{given.t} 
is (for $C$ as in the claim) less than
\beq{pathest} 
C^{-k/3}|\g | {2n \choose t} {m-2n \choose {w'}}.
\enq
To see this we think of choosing $(S,W\cap S)$---which can be 
done in at most
$    %\[
|\g|{2n \choose t}
$  %\]
ways---and then $W\sm S$.
For the latter, notice that $(S,W)$ bad means that \emph{every} $J\sub S\cup W $ 
%($= S\cup (W\sm S)$)
has $|J\cap S|$ ($\ge |J\sm W|$) $\ge k$; so, since $(S,W)$ is pathological,
\[
|\{J \subseteq S\cup (W\sm S):|J \cap S| \ge k\}| \ge C^{-k/3} |\g | \C{w'+2n}{2n}\Big/\C{m}{2n}.
\]
But then Claim~\ref{claim.patho} (with Markov's Inequality)
says the number of possibilities for $W\sm S$ is at most
\[
 C^{-k/3} {m-2n \choose {w'}}.
\]
Thus we have \eqref{pathest} and combining with \eqref{nonpathest}  completes the 
proof of Lemma~\ref{MP}.  
\end{proof}

\begin{proof}[Proof of Claim~\ref{claim.patho}]
With $f_i$ the fraction of $J$'s (in $\g$) with $|J\cap S|=i$, the left-hand side of~\eqref{patho} is 
\[ 
\sum_{i \ge k} |\g| f_i {w' \choose 2n-i}\Big/{m-2n \choose 2n-i},
\]
so it is enough to show
\beq{fiand}
f_i \left[ {w' \choose 2n-i}\Big/{m-2n \choose 2n-i}\right]
\left[{w'+2n \choose 2n}\Big/{m \choose 2n}\right]^{-1} = e^{O(i)}C^{-i},
\enq
where---here and below---implied constants do not depend on $C$.
The terms other than $f_i$ on left-hand side of \eqref{fiand} reduce to
\[
 \frac{(w')_{2n-i}}{(w'+2n)_{2n-i}} \cdot \frac{(m)_{2n-i}}{(m-2n)_{2n-i}} \cdot \frac{(m-2n+i)_i}{(w'+i)_i} 
 %= O(1) \cdot e^{O(i)} C^{-i} n^{i/2} 
 =   e^{O(i)} C^{-i} n^{i/2}
\]
(we omit the routine calculation, just noting that $\sqrt{n}=O(i)$ since
$i\ge k$), so for \eqref{fiand} we just need
\beq{goal}
f_i \le e^{O(i)} n^{-i/2}.
\enq

For $n/3 \le i \le 2n$, this follows from the fact that $\g$ is $q$-spread with $q \sim \sqrt{e/n}$
(see \eqref{spread}), which gives
\[f_i \le {2n \choose i} q^i =e^{O(i)}n^{-i/2}.\]

For $k \le i \le n/3$, Propositions~\ref{easy} and~\ref{easy2} (with Stirling's formula) give
\[
f_i ~\le ~|\g|^{-1}(128e)^i\sum_{c=1}^i  {4n \choose c} \left(n-\left\lceil \frac{i+c}{2}\right\rceil {-1} \right)! 
~= ~e^{O(i)} n^{-i/2}\sum_{c=1}^i (\sqrt{n}/c)^c~=~ e^{O(i)} n^{-i/2},
\]
where at the end we use
$(a/x)^x \le e^{a/e}$ and $i \ge k $.\end{proof}

\begin{proof}[Proof of Theorem~\ref{mt2}]
We prove this for $K  = 3C_0+C$, with $C_0$ as in Lemma~\ref{MP}
and $C$ a suitable function of $\eps$ (essentially $1/\eps$).
Let $p_0=3C_0/\sqrt{n}$, $p_1=C/\sqrt{n}$ and $p=p_0+p_1-p_0p_1 <K/\sqrt{n}$.
We generate $G_{n,p}$ in two rounds, as $W_0\cup W_1$, where 
$W_0$, $W_1$ are independent with $W_\nu$ distributed as $G_{n,p_\nu}$
(and $W_1$ chosen after $W_0$, at which point we are really interested in $W_1\sm W_0$).

Call $W_0$ \emph{successful} if
\[
|\{S: (S,W_0) ~\text{is bad}\}| \leq |\g|/2.
\]
We first observe that $W_0$ is (very) likely to be successful:  standard concentration estimates give (say)
\[
\pr(|W_0|< C_0n^{3/2}) = \exp[-n^{3/2}] ,
\]
and Lemma~\ref{MP} gives
\[
\pr(W_0 ~\text{unsuccessful}\,| \,|W_0|\geq C_0n^{3/2}) < 4 C_0^{-k/3};
\]
in particular $W_0$ is successful with probability $1-o(1)$.

Suppose now that $W_0$ is successful. 
For each $S$ with $(S,W_0)$ good, let 
$\chi(S,W_0)$ be some $k$-element subset of $S$ containing
an $(S, W_0)$-fragment, and let $\R$ be the $k$-uniform (\emph{multi})hypergraph
\beq{chi}
\{\chi(S,W_0): (S, W_0) \text{ is good}\}.
\enq

To finish
we use the second  moment method to show that
$W_1$ is reasonably likely to contain a member of $\R$. 
Setting
\[
X = |\{A\in \R:A\sub W_1\}|,
\]
we have
\[ 
\mu :=  \E X = |\R| p_1 ^ k 
\]
and 
\beq{var}
{\rm Var} (X) \leq 
%\sum_{A,B\in \R,A\cap B\neq \0} \pr(W_1\supseteq A\cup B) = 
p_1^{2k}\sum\{p_1^{-|A\cap B|}:A,B\in \R,A\cap B\neq \0\}.
\enq
For $R \in \R$ and $1\leq i \le k$, Propositions~\ref{easy} and~\ref{easy2} (with Stirling) give
\begin{eqnarray*}
%\beq{rsum} 
|\{A\in \R:|A\cap R|=i\}| &\le &
\sum_{I \subseteq R, |I| = i} |\R \cap \langle I \rangle|
\leq
\sum_{I \subseteq R, |I| = i} |\g \cap \langle I \rangle|\\
&=& e^{O(i)}\sum_{1\le c \le i} \C{2k}{c} \left(n-\left\lceil {\frac{i+c}{2}} \right\rceil {-1} \right)! 
= e^{O(i)} n^{-i/2}|\g|;
\end{eqnarray*}
so (recall $W_0$ successful means $|\R|\geq |\g|/2$)
the sum in \eqref{var} is at most 
\[ 
2|\R|^2p_1^{2k}\sum_{i=1}^k e^{O(i)} p_1^{-i} n^{-i/2}
=O(\mu^2/C)
\]
for large enough $C$ (where, again, the implied constant doesn't depend on $C$).
Thus, finally, Chebyshev's Inequality gives
\[
\pr(X =0) \le {\rm Var}(X)/\mu^2 =O(1/C),
\]
and we are done.
\end{proof}

%%%%%%%%%%%%%%%%%
\iffalse
\section{Conclusion}\label{s:conc}
Our work still leaves open a natural question, namely that of locating the sharp threshold for $G_{n,p}$ to contain a square of the Hamilton cycle. The  first moment calculation presented earlier suggests the following natural guess.
\begin{conjecture}
For any fixed $\eps > 0 $, if $p \ge (1+ \eps)\sqrt{e  / n}$, then
\[ \pr (G_{n,p} \text{ contains the square of Hamilton cycle}) \to 1\]
as $n \to \infty$.
\end{conjecture}
\fi
%%%%%%%%%%%%%%%%%%%

\section*{Acknowledgments}
The first and second authors were supported by NSF grants DMS1954035
and DMS-1800521 respectively.
The third author's work was supported directly by NSF
grant DMS-1926686
and indirectly by NSF grant CCF-1900460.

\bibliographystyle{amsplain}
\bibliography{201016_hc2}

\end{document}